\documentclass[hidelinks,onefignum,onetabnum]{siamart250211}

\usepackage{amssymb,amsfonts,amsmath,mathrsfs,amsopn}
\usepackage{graphicx}
\usepackage{subcaption}
\usepackage{cleveref}
\usepackage{algorithm}
\usepackage{algpseudocode}
\usepackage{xcolor}

\ifpdf
  \DeclareGraphicsExtensions{.eps,.pdf,.png,.jpg}
  \hypersetup{
    pdftitle={Efficient Long-Time Simulations of Multiscale Systems via High-Order Numerical Homogenization},
    pdfauthor={B. Chen, Z. Jin, and R. Li}
  }
\else
  \DeclareGraphicsExtensions{.eps}
\fi

\newsiamremark{assumption}{Assumption}
\newsiamremark{example}{Example}
\newsiamremark{remark}{Remark}
\newsiamremark{experiment}{Experiment}
\newsiamremark{hypothesis}{Hypothesis}
\crefname{hypothesis}{Hypothesis}{Hypotheses}
\crefname{assumption}{Assumption}{Assumptions}
\crefname{equation}{equation}{equations}
\crefname{figure}{figure}{figures}
\crefname{table}{table}{tables}
\crefname{section}{section}{sections}
\crefname{appendix}{appendix}{appendices}
\crefname{theorem}{theorem}{theorems}
\crefname{lemma}{lemma}{lemmas}
\crefname{corollary}{corollary}{corollaries}
\crefname{proposition}{proposition}{propositions}
\crefname{definition}{definition}{definitions}
\crefname{remark}{remark}{remarks}
\crefname{note}{note}{notes}

\newsiamthm{claim}{Claim}

\newcommand{\bR}{\mathbb{R}}

\newcommand{\cC}{\mathcal{C}}
\newcommand{\cE}{\mathcal{E}}
\newcommand{\cG}{\mathcal{G}}
\newcommand{\cO}{\mathcal{O}}
\newcommand{\eps}{\varepsilon}
\newcommand{\rd}{{\mathrm{d}}}

\newcommand{\rx}{\mathbb{R}^{n_x}}
\newcommand{\ry}{\mathbb{R}^{n_y}}
\newcommand{\tx}{\tilde{x}}
\newcommand{\Cost}{\text{Cost}}
\newcommand{\HMM}{{\text{HMM}}}

\newcommand{\Gronwall}{{Gr\"onwall}}

\newcommand{\norm}[1]{\left\|#1\right\|}
\newcommand{\sinn}[1]{\langle{#1}\rangle}

\newcommand{\ip}[2]{\left\langle{#1},{#2}\right\rangle}
\newcommand{\od}[2]{\dfrac{\mathrm{d} #1}{\mathrm{d} #2}}

\graphicspath{{fig/}}

\headers{Efficient Long-Time Simulations of Multiscale Systems}
{Bojin Chen, Zeyu Jin and Ruo Li}

\title{Efficient Long-Time Simulations of Multiscale Systems via High-Order Numerical Homogenization \thanks{Submitted to the editors DATE.
\funding{Ruo Li is partially supported by the National Natural Science Foundation of China (Grant No. 12288101).}}}

\author{Bojin Chen\thanks{School of Mathematical Sciences, 
Peking University, Beijing 100871, China  
  (\email{2501110041@stu.pku.edu.cn}).}
\and Zeyu Jin\thanks{School of Mathematical Sciences, 
Peking University, Beijing 100871, China  
  (\email{jinzy@pku.edu.cn}).}
\and Ruo Li\thanks{CAPT, LMAM and School of Mathematical Sciences, 
Peking University, Beijing 100871, China;  
Chongqing Research Institute of Big Data, Peking University, 
Chongqing 401121, China (\email{rli@math.pku.edu.cn}).}}

\begin{document}
\maketitle

\begin{abstract}
  By a high-order numerical homogenization method, a heterogeneous
  multiscale scheme was developed in \cite{Jin2022HighOrder} for
  evolving differential equations containing two time scales. In this
  paper, we further explore the technique to propose an efficient
  algorithm able to carry out simulations up to a long time which was 
  prohibitive before. The new algorithm is a multigrid-in-time method
  which combines coarse-grid high-order approximations with fine-grid
  low-order evaluations. The high efficiency is attained by minimizing
  the computational cost while the approximation accuracy is
  guaranteed. A priori error estimates are rigorously established and
  validated by numerical examples.
\end{abstract}



\section{Introduction}
Multiscale phenomena are ubiquitous in science and engineering, arising in 
fields including celestial mechanics \cite{Laskar1994Large},
climate dynamics \cite{MR2922369}, 
materials science \cite{elliott2011novel,abraham1997molecular}, and 
molecular dynamics \cite{Car1985Unified,Zhang1999A}. 
Such systems are characterized by the intricate interplay between rapidly 
evolving processes, often driven by strong dissipation or high-frequency 
oscillations \cite{Weinan2003Analysis}, and slowly varying modes that 
dictate the long-term behavior of the system.
The core difficulty in their numerical treatment is the severe stiffness 
introduced by the disparate time scales.
Direct numerical simulation that resolves all scales over long time 
horizons is computationally expensive, as stability requirements for 
explicit methods force a time step on the order of the fastest scale, 
leading to a prohibitive number of computations.
Moreover, naive coarse-graining approaches can introduce significant 
biases over long-time scales, potentially obscuring the emergent behaviors 
of interest.

Various numerical strategies have been proposed to mitigate stiffness,
including comprehensive surveys of stiff integrators 
\cite{Hairer1996Solving}, implicit-explicit (IMEX) methods 
\cite{ascher1997implicit,frank1997stability}, and exponential integrators 
\cite{hochbruck2010exponential}.
However, these methods often do not fully exploit the scale separation 
inherent in the problem.
When only a limited set of macroscopic observables is of interest, a more 
effective strategy is to leverage this scale separation by sampling the 
fast dynamics only as needed.
Prominent among methods that leverage scale separation are the 
equation-free approach \cite{MR2041455} and 
the Heterogeneous Multiscale Method (HMM)
\cite{Weinan2012Review,Weinan2012The}.
The HMM framework provides a general and powerful strategy for coupling a 
macroscopic integrator for the slow variables with localized microscopic 
solvers that estimate the effective forces required by the macro-solver
\cite{Weinan2003Analysis}.
This macro-micro coupling allows for the simulation of the slow dynamics 
with large time steps that are independent of the fast scales, leading to 
significant computational savings.

Despite its success, the classical HMM framework faces a critical accuracy 
limitation when applied to long-time simulations.
The standard approach introduces a first-order modeling error 
\cite{Weinan2003Analysis} in the scale-separation parameter, $\eps$, 
of order $\cO(\eps)$. 
Under certain stability assumptions, this error accumulates linearly over 
time. 
Consequently, for simulation horizons of interest, such as 
$t \sim \eps^{-1}$, the accumulated modeling error becomes $\cO(1)$, 
rendering the classical HMM inadequate for producing accurate numerical 
results.

To address this accuracy bottleneck of the classical HMM, a high-order 
numerical homogenization method was developed in \cite{Jin2022HighOrder}.
The approach, based on a fixed-point iteration algorithm, numerically 
computes an asymptotic expansion of the system's slow manifold 
\cite{MR635782,MR715971,MR2759609}, 
achieving a modeling error of $\cO(\eps^{k+1})$ for the $k$-th order model 
over any fixed time interval.
It enables high-order accuracy without requiring manual derivations of 
complicated asymptotic expansions or the storage of high-order tensors.
However, this advance, while addressing the accuracy problem for a 
finite-time simulation, introduced two new critical challenges, which form 
the primary motivation for the present paper.

First, a theoretical gap remained concerning the long-time error behavior. 
The error analysis provided in \cite{Jin2022HighOrder} yields 
bounds that grow exponentially with time, of the form $e^{C_k t}$. 
While sufficient for fixed, finite time intervals, this exponential 
factor fails to guarantee accuracy over the extended horizons relevant 
to many applications, such as $t \sim \eps^{-1}$. 
A rigorous theory for long-time error accumulation was needed.
Second, an efficiency gap emerged from the method's construction. 
The recursive numerical construction of the $k$-th order 
correction model, while elegant, is computationally expensive. 
Each successive order of correction roughly doubles the number of 
microscopic solver evaluations required, leading to a total computational 
complexity that scales as $\cO(2^k)$ relative to the classical HMM. 
This redundancy in cost renders high-order models impractical for the very 
long-time simulations where their accuracy is most needed.

In this work, these gaps are bridged with a dual contribution in theory 
and algorithmic design.
First, a rigorous long-time error analysis for the high-order HMM is 
provided.
It is proved that the modeling error accumulates linearly in time, not 
exponentially, as long as the problem is well-posed for long-time 
simulations. 
This yields a modeling error bound of $\cO(\eps^{k+1} t)$, which provides 
a solid theoretical foundation for the use of high-order HMM in long-time 
regimes. 
Second, a novel multigrid-in-time (MGT) HMM algorithm is introduced.
This framework achieves the accuracy of a high-order model at a 
computational cost comparable to that of a low-order baseline method.
The core idea is to use hierarchical temporal discretization, computing  
expensive high-order corrections only on coarse temporal grids and 
propagating their effect to finer grids via extrapolation, thereby 
decoupling high accuracy from high computational cost.

The remainder of the paper is organized as follows. 
\Cref{sec:pre} establishes the mathematical framework, reviewing rapidly 
dissipative fast-slow systems and the high-order HMM.
\Cref{sec:model} derives the improved long-time modeling error bounds for 
the high-order models.
\Cref{sec:grid} presents the multigrid-in-time algorithm in detail, 
including its complexity analysis.
\Cref{sec:analysis} provides a rigorous a priori numerical error analysis 
of the complete MGT-HMM scheme.
\Cref{sec:num} provides a comprehensive suite of numerical experiments 
validating our theoretical claims on several benchmark problems.
Finally, \Cref{sec:concl} offers some conclusive remarks.

\section{Preliminaries}
\label{sec:pre}
This section establishes the mathematical framework for the class of 
problems under consideration and reviews the high-order numerical 
homogenization method that forms the basis of this work.

\subsection{Rapidly dissipative fast-slow systems}
\label{sec:pre_dissipative}
Let us consider the following fast-slow system
\cite{pavliotis2008multiscale}:
\begin{equation}
  \label{equation}
  \left\{
    \begin{aligned}
      & \od{x}{t} = f(x,y), \ & x(0) = x_0, \\
      & \od{y}{t} = \frac{1}{\eps}  g(x,y), \ & y(0) = y_0,
    \end{aligned}
  \right.
\end{equation}
where $x \in \rx$ and $y \in \ry$ are the slow and fast variables, 
respectively, and $0< \eps \le \eps_0 \ll 1$ is the scale-separation 
parameter.

In the subsequent discussions, we always impose the following regularity 
conditions for the right-hand term of the system \eqref{equation}.

\begin{assumption}[Regularity]
  \label{assump:bound}
  The functions $f$ and $g$ are sufficiently smooth, with bounded 
  derivatives up to some order:
  \begin{displaymath}
    f \in W^{K,\infty}(\rx \times \ry, \rx), \quad \nabla g \in
    W^{K,\infty}(\rx \times \ry, \bR^{n_y \times (n_x + n_y)})
  \end{displaymath}
  for some integer $K > 0$.
\end{assumption}

Furthermore, the system \eqref{equation} is assumed to be 
\emph{rapidly dissipative}.
\begin{assumption}[Rapid dissipation]
  \label{assump:dissipative}
  There exists $\beta > 0$ such that for all $x \in \rx$ and 
  $y, \tilde{y} \in \ry$, one has that 
  \[
    \ip{g(x,y) - g(x,\tilde{y})}{y - \tilde{y}} 
    \le -\beta |y-\tilde{y}|^2.
  \]
\end{assumption}

Under these assumptions, it is shown in \cite{Jin2022HighOrder} that 
$g(x, \cdot) : \ry \to \ry$ is bijective for each fixed $x \in \rx$, 
yielding a unique $\gamma(x) \in \ry$ such that $g(x, \gamma(x)) = 0$.
This ensures that for each fixed $x \in \rx$, the fast dynamics,
\begin{equation}\label{equ:yx}
  \od{y^x}{t} = \frac{1}{\eps} g(x,y^x), 
\end{equation}
is globally exponentially contracting \cite{pavliotis2008multiscale}, 
so it rapidly relaxes to the unique equilibrium $\gamma(x)$.
This property ensures the existence of a slow manifold \cite{MR635782}, 
which the system's trajectory quickly approaches.
In the limit $\eps \to 0$, the slow variables $x(t)$ converge to 
the solution of the \textit{reduced system}:
\begin{equation}\label{equ:limiting_equation}
  \od{X_0}{t}=f(X_0, \gamma(X_0)), \quad X_0(0) = x_0.
\end{equation}
This approximation introduces an $\cO(\eps)$ modeling error over any fixed 
finite time interval \cite{Papanicolaou1976Some,pavliotis2008multiscale}.

For long-time simulation to be a well-posed problem, the underlying system 
must be stable with respect to small perturbations over the time scales of 
interest.
This leads to the next assumption, which concerns the stability of the 
true physical system.

\begin{assumption}[Long-time stability]\label{assump:stability}
Let $\Phi_t^\eps(x_0, y_0) = (x(t), y(t))$ denote the flow map of the full 
system \eqref{equation}.
The system is assumed to be \emph{stable} up to the 
\emph{maximal stability time} $T_\eps$ in the sense that there exists a 
constant $A > 0$ independent of $\eps$ and $t$, such that the Jacobian of 
the flow map is uniformly bounded on the time scale of interest $T_\eps$:
\[
\| \nabla \Phi_t^\eps \|_{L^\infty} \le A, 
\]
for each $t \in [0, T_\eps]$.
\end{assumption}

This is not merely a technical convenience but a physical prerequisite.
If a system is inherently unstable, such that small initial perturbations 
lead to exponentially diverging trajectories, no numerical method can 
provide meaningful long-time point-wise predictions.

To simplify the analysis of error accumulation, a more restrictive 
technical assumption on the reduced vector field is also considered.

\begin{assumption}[Mild growth]\label{assump:growth}
The reduced vector field $F_0(x) = f(x, \gamma(x))$ satisfies that 
\begin{equation}\label{equ:growth}
\lambda_0 = \sup_{x, \tx \in \rx \atop x \ne \tx} 
\frac{\ip{F_0(x) - F_0(\tilde{x})}{x - \tilde{x}}}{|x - \tx|^2} \le 0.
\end{equation}
\end{assumption}

This assumption is a stronger condition than \Cref{assump:stability}.
It prevents the exponential growth of perturbations in the effective slow 
dynamics and facilitates a more direct proof of the stability of the 
high-order numerical models themselves. 
See further discussions in \Cref{sec:model_cond}.

\subsection{The heterogeneous multiscale method (HMM) and its high-order 
extension}
\label{sec:pre_hmm}
The classical HMM \cite{Weinan2012The,Weinan2003Analysis} provides a 
strategy to integrate the slow variables without resolving the fast 
dynamics at every step.
It employs a macroscopic solver to evolve the slow variables $x$ according 
to the reduced system \eqref{equ:limiting_equation}, and a microscopic 
solver to estimate the effective force $F_0(x) = f(x, \gamma(x))$.
The micro-solver essentially solves the nonlinear equation 
$g(x, y) = 0$ for fixed $x$, which can be achieved, for example, by 
evolving the fast dynamics \eqref{equ:yx} over a short time window.

To improve accuracy, the high-order numerical homogenization method from  
\cite{Jin2022HighOrder} seeks higher-order approximations of the slow 
manifold. 
The \textit{$k$-th order correction model}, denoted by $\HMM_k$, is given 
as follows:
\begin{equation}
    \label{equ:highmodel}
    \od{X_k}{t} = f(X_k, \Gamma_k(X_k, \eps)), \quad X_k (0) = x_0,
\end{equation}
where $\Gamma_k$ is a sequential fixed-point approximation of the 
invariant manifold defined by the recursive relation:
\begin{equation}\label{eq:iteration}
  g(x, \Gamma_{k+1}(x, \eps)) = \eps \nabla_x \Gamma_k(x, \eps) \cdot 
  f(x, \Gamma_k(x, \eps)), 
\end{equation}
with the initial value $\Gamma_0(x,\eps) = \gamma(x)$.
Numerically, the right-hand side of \eqref{eq:iteration} can be computed 
via a finite difference approximation of the directional derivative of 
$\Gamma_k(\cdot, \eps)$ along the direction $f(x, \Gamma_k(x, \eps))$
according to \cite{Jin2022HighOrder}:
\begin{equation}\label{eq:numder}
\nabla_x \Gamma_k(x, \eps) \cdot f(x, \Gamma_k(x, \eps)) = 
\lim_{\eta \to 0} \frac{\Gamma_k(x + \eta f(x, \Gamma_k(x,\eps)), \eps) - 
\Gamma_k(x, \eps)}{\eta},
\end{equation}
and then $\Gamma_{k+1}(x, \eps)$ can be calculated by the micro-solver 
by inverting the function $g(x, \cdot)$.

The following lemma from \cite{Jin2022HighOrder}, which plays a central 
role in our analysis, establishes the successive approximation properties
of this construction.
\begin{lemma}\label{lemma:asym}
For $k = 0, 1, \ldots, K$, 
\begin{displaymath}
\norm{\Gamma_{k+1}(\cdot, \eps) - \Gamma_k(\cdot, \eps)}_{C^{K- k}} 
\lesssim \eps^{k+1}.
\end{displaymath}
\end{lemma}

Based on this, it was shown in \cite[Theorem 3.9]{Jin2022HighOrder} that 
the manifold $\{ y = \Gamma_k(x, \eps) \}$ is a global attractor of 
\eqref{equation} with error $\cO(\eps^{k+1})$: 
\begin{equation}\label{eq:attractor}
  |z_k(t)| \lesssim \eps^{k+1} + e^{-\beta t / \eps} |z_k(0)|, \quad 
  z_k(t) = y(t) - \Gamma_k(x(t), \eps).
\end{equation}
In addition, after an initial layer, the $k$-th correction model 
\eqref{equ:highmodel} can approximate the original system \eqref{equation} 
up to an error of order $\cO(\eps^{k+1})$ on any finite time horizon: 
\begin{equation}\label{eq:decouple}
  |x(t) - X_k(t)|^2 \lesssim e^{C_k t} \left( |x(0) - X_k(0)|^2 + 
  \eps^{2k+2} + \eps \beta^{-1} |z_k(0)|^2 \right),
\end{equation}
according to \cite[Theorem 3.10]{Jin2022HighOrder}.
Here the exponential term $e^{C_k t}$ in the bound \eqref{eq:decouple} 
prevents an available long-time error estimate, which is the primary 
analysis gap we aim to address in the next section.

\section{Long-time modeling error analysis} 
\label{sec:model}
This section addresses the analysis gap identified in the previous section 
by deriving a new modeling error bound for the high-order correction 
models.
The central result demonstrates linear, rather than exponential, error 
accumulation over long time horizons, providing the theoretical 
justification for using these methods for long-time simulations.

\subsection{A linear error accumulation bound}
\label{sec:model_longtime}

The first main theoretical result of this paper establishes that the 
modeling error accumulates linearly in time within the maximal interval of 
stability defined by \Cref{assump:stability}.

\begin{theorem}\label{thm:longtime}
Let $T_\eps$ be the maximal stability time of \eqref{equation} 
according to \Cref{assump:stability}.
If $X_k(0) = x(0)$ and the initial state is off the slow manifold by at 
most $|z_k(0)| = |y(0) - \Gamma_k(x(0), \eps)| = \cO(\eps^{k+1})$, 
then there exists a constant $C > 0$ independent of $\eps$, $t$, and $A$
such that 
\begin{displaymath}
|x(t) - X_k(t)| \le C A \eps^{k+1} (t + 1),
\end{displaymath}
for each $t \in [0, T_\eps]$ and $k \le K - 1$.
\end{theorem}

\begin{proof}
Fix the final time $t \in [0, T_\eps]$.
Define an auxiliary trajectory $(u(s), v(s))$ by evolving the full system 
in time from the state of the high-order model at time $s$:
\[(u(s), v(s)) = \Phi_{t-s}^\eps (X_k(s), \Gamma_{k+1}(X_k(s), \eps)),
\quad s \in [0, t].
\]
At the endpoints, one has 
\begin{gather*}
(u(0), v(0)) = \Phi_{t}^\eps (X_k(0), \Gamma_{k+1}(X_k(0), \eps)), \\
(u(t), v(t)) = (X_k(t), \Gamma_{k+1}(X_k(t), \eps)).
\end{gather*}
According to the flow relation, the derivative of this trajectory with 
respect to $s$ is given by
\[
(u'(s), v'(s)) = \nabla \Phi_{t-s}^\eps (X_k, \Gamma_{k+1}) \cdot 
\begin{pmatrix}
f(X_k, \Gamma_{k}) - f(X_k, \Gamma_{k+1}) \\
\nabla \Gamma_{k+1} \cdot f(X_k, \Gamma_k) - 
\frac1\eps g(X_k, \Gamma_{k+1})
\end{pmatrix}.
\]
Here we omit the independent variable $s$ for brevity.
Using the definition of $\Gamma_{k+1}$ in \eqref{eq:iteration}, one has 
that 
$g(X_k, \Gamma_{k+1}) = \eps \nabla_x \Gamma_k \cdot f(X_k, \Gamma_k)$.
By \Cref{lemma:asym} and the smoothness of $f$ in \Cref{assump:bound}, the 
norms of these components with $k \le K - 1$ can be bounded by
\begin{gather*}
|f(X_k, \Gamma_k) - f(X_k, \Gamma_{k+1})| \le \| \nabla_y f \|_{C^0}
\cdot |\Gamma_k - \Gamma_{k+1}| \lesssim \eps^{k+1}, \\
|(\nabla \Gamma_{k+1} - \nabla \Gamma_k) \cdot f(X_k, \Gamma_k)| \le 
\|\nabla \Gamma_{k+1} - \nabla \Gamma_k\|_{C^0} \cdot \|f\|_{C^0} 
\lesssim \eps^{k+1}.
\end{gather*}
Using the stability of the flow map from \Cref{assump:stability}, 
$\|\nabla \Phi_{t-s}^\eps\| \le A$, one can bound the rate of change of 
$u(s)$ by $|u'(s)| \lesssim A \eps^{k+1}$.
Integrating this from $s = 0$ to $s = t$ yields the accumulated error over 
the interval:
\[
|u(t) - u(0)| \le \int_0^t |u'(s)| \,\rd s \lesssim A \eps^{k+1} t.
\]
The error from the initial conditions can be bounded using 
\Cref{assump:stability}:
\begin{multline*}
|x(t) - u(0)| \\ \le |\Phi_t^\eps(x_0, y_0) - 
\Phi_t^\eps(X_k(0), \Gamma_{k+1}(X_k(0), \eps))| \lesssim 
A |y_0 - \Gamma_{k+1}(x_0, \eps)| \lesssim A \eps^{k+1}.
\end{multline*}
Finally, combining the two parts via the triangle inequality, 
one obtains that 
\[
|x(t) - X_k(t)| \le |x(t) - u(0)| + |u(t) - u(0)| \lesssim
A \eps^{k+1} (t + 1),
\]
which yields the final result.
\end{proof}

\begin{remark}
The condition, $|z_k(0)| = \cO(\eps^{k+1})$, means that the initial layer 
effect is ignored. 
In practice, the full system is numerically evolved in the initial layer, 
which has a thickness of order $\cO(\eps |\log \eps|)$ as shown in 
\eqref{eq:attractor}, after which this condition is reasonably satisfied.
See discussions in \cite[Remark 3.6]{Jin2022HighOrder}.
\end{remark}

\begin{remark}
The proof sidesteps a direct \Gronwall-type analysis, which leads to the 
exponential term, by employing an auxiliary trajectory and leveraging the 
global stability of the flow map. 
This approach more accurately reflects the error propagation mechanism in 
well-posed long-time problems.
\end{remark}

\subsection{Discussion of stability and error propagation}
\label{sec:model_cond}

The proof of \Cref{thm:longtime} relies only on the general well-posedness 
condition in \Cref{assump:stability}.
If the more restrictive mild growth condition in \Cref{assump:growth} 
holds, a more refined error estimate can be obtained via a more direct 
analysis.

A crucial element for long-time error control is that the high-order 
correction models \eqref{equ:highmodel} inherit the stability properties 
of the underlying reduced system \eqref{equ:limiting_equation} under 
\Cref{assump:growth}.
Let $F_k(x) = f(x, \Gamma_k(x, \eps))$ be the vector field of the $k$-th 
order model and define its one-sided Lipschitz constant $\lambda_k(\eps)$
analogously to $\lambda_0$ in \eqref{equ:growth}:
\[
\lambda_k(\eps) = \sup_{x, \tx \in \bR^{n_x} \atop x \ne \tx} 
\frac{\sinn{F_k(x) - F_k(\tx), x - \tx}}{|x - \tx|^2}.
\]
The difference between successive constants can be bounded by: 
\begin{multline*}
|\lambda_{k+1}(\eps) - \lambda_k(\eps)| \\
\le \sup_{x, \tx \in \rx \atop x \ne \tx} 
\frac{|\sinn{F_{k+1}(x) - F_{k+1}(\tx)-F_{k}(x) + F_{k}(\tx), x - \tx}|}
{|x - \tx|^2} \lesssim \| \nabla F_{k+1} - \nabla F_k \|_{C^0},
\end{multline*}
where $k \le K - 1$.
By the chain rule, 
$\nabla F_k = \nabla_x f + \nabla_y f \cdot \nabla_x \Gamma_k$.
Thus, the difference is bounded by $\|f\|_{C^2} \cdot 
\| \Gamma_{k+1}(\cdot, \eps) - \Gamma_k(\cdot, \eps) \|_{C^1}$.
Using \Cref{lemma:asym}, one obtains that 
\[
|\lambda_{k+1}(\eps) - \lambda_k(\eps)| \lesssim
\| \Gamma_{k+1}(\cdot, \eps) - \Gamma_k(\cdot, \eps) \|_{C^1} 
\lesssim \eps^{k+1}.
\]
This implies that the stability constant $\lambda_k(\eps)$ of the $k$-th 
order model is a small perturbation of $\lambda_0$:
$
\lambda_k(\eps) = \lambda_0 + \cO(\eps).
$
By \Gronwall's inequality, the separation of two trajectories of the 
$k$-th order model is amplified by $e^{\lambda_k(\eps) t}$:
\begin{displaymath}
|X_k(t) - \tilde{X}_k(t)| \le e^{\lambda_k(\eps) t} 
|X_k(0) - \tilde{X}_k(0)|.
\end{displaymath}
This, together with \Cref{assump:growth}, ensures that 
$\lambda_k(\eps) \lesssim \eps$, which guarantees that the maximal 
stability time of the high-order correction models \eqref{equ:highmodel} 
is at least of order $\eps^{-1}$.
This confirms that the correction models are themselves well-posed for 
long-time integration, and leads to the following refined long-time 
modeling error estimate under \Cref{assump:growth}.

\begin{theorem}\label{thm:longtime_lambda}
Under the same conditions as in \Cref{thm:longtime} and with 
\Cref{assump:growth}, there exists a constant $C > 0$ independent of 
$\eps$ and $t$ such that 
\[
|x(t) - X_k(t)| \le C \eps^{k+1} \zeta_k(\eps, t),
\]
where $\zeta_k(\eps, t) = \int_0^t e^{\lambda_k(\eps) s} \,\rd s$.
Furthermore, the maximal stability time $T_\eps$ of the high-order models 
\eqref{equ:highmodel} is determined by 
$\lambda_k(\eps)$: $T_\eps \sim \lambda_k(\eps)^{-1}$ if 
$\lambda_k(\eps) > 0$, or $T_\eps = \infty$ if $\lambda_k(\eps) \le 0$.
The behavior of the error growth factor $\zeta_k$ depends on the sign and 
magnitude of $\lambda_k(\eps)$ as follows: 
\begin{itemize}
\item If $\lambda_k(\eps) \lesssim \eps^m$, then 
$\zeta_k(\eps, t) \lesssim t$ up to $T_\eps \sim \eps^{-m}$.
\item If $\lambda_k(\eps) \le 0$, then $\zeta_k(\eps, t) \le t$ up to 
$T_\eps = \infty$.
\item If $\lambda_k(\eps) \lesssim - \eps^m$, then 
$\zeta_k(\eps, t) \lesssim \min \{t, \eps^{-m}\}$ up to $T_\eps = \infty$.
\end{itemize}
\end{theorem}

\begin{proof}
Differentiating $|x - X_k|^2 / 2$ and applying \eqref{eq:attractor} leads 
to the following inequality:
\begin{multline*}
\frac12 \od{}{t} |{x - X_k}|^2 = 
\sinn{x - X_k, f(x,y) - f(X_k, \Gamma_k(X_k,\eps))}\\ = 
\sinn{x - X_k, f(x,y) - f(x, \Gamma_k(x, \eps))} + 
\sinn{x - X_k, f(x, \Gamma_k(x,\eps)) - f(X_k, \Gamma_k(X_k, \eps))} \\ 
\le C |x - X_k| \cdot |z_k(t)| + \lambda_k(\eps) \cdot |x - X_k|^2 
\le C \eps^{k+1} |x - X_k| + \lambda_k(\eps) \cdot |x - X_k|^2,
\end{multline*}
which yields $|x(t) - X_k(t)| \lesssim \eps^{k+1} \zeta_k(\eps, t)$.
If $\lambda_k(\eps, t) \lesssim \eps^m$, 
then stability holds up to $T_\eps \sim \eps^{-m}$, and for 
$t \in [0, T_\eps]$, one has that 
\[
\zeta_k(\eps, t) \le \int_0^t e^{c \eps^m s} \,\rd s 
\lesssim \eps^{-m} (e^{c \eps^m t} - 1) \lesssim t.
\]
If $\lambda_k(\eps) \le 0$, then $\zeta_k(\eps, t) \le t$.
If $\lambda_k(\eps) \lesssim - \eps^m$, then 
\[
\zeta_k(\eps, t) \le \eps^{-m} (1 - e^{-c \eps^m t}) \lesssim
\begin{cases}
t,\ & t \lesssim \eps^{-m}, \\
\eps^{-m},\ & t \gtrsim \eps^{-m}.
\end{cases}
\]
This completes the proof.
\end{proof}

\subsection{Numerical illustration of error accumulation}
\label{sec:model_failure}

To illustrate the failure of low-order methods and to validate the 
theoretical predictions in \Cref{thm:longtime}, let us consider a simple 
linear example with combined phase and amplitude drifts:
\begin{equation}\label{equ:model_failure}
\left\{
\begin{aligned}
&\od{x}{t} = x + y, \\
&\od{y}{t} = \frac1\eps \left( - x + J x - y\right), \\
\end{aligned}    
\right.
\end{equation}
where $x, y \in \bR^2$, and $J = \left( \begin{smallmatrix}
0 & 1 \\ -1 & 0 \end{smallmatrix} \right)$.
This system is analytically tractable, allowing for precise comparison.
The four eigenvalues $\alpha_{\pm,\pm}$ of \eqref{equ:model_failure} are 
given by 
\begin{gather*}
\alpha_{-,\pm} = \frac12 \left( -\mu_\pm - \eps^{-1} + 1 \right) = 
-\eps^{-1} + \cO(1), \\
\alpha_{+,\pm} = \frac12 \left( \mu_\pm - \eps^{-1} + 1 \right) = 
\pm i + (1 \pm i) \eps + (3 \mp i) \eps^2 + (1 \mp 9 i) \eps^3 + 
\cO(\eps^4),
\end{gather*}
where $\mu_\pm = \sqrt{\eps^{-2} + (- 2 \pm 4i) \eps^{-1} + 1}$.
The reduced and higher-order models for this system can be derived 
analytically:
\[
\od{X_0}{t} = J X_0, \quad 
\od{X_1}{t} = \eps X_1 + (1+\eps) J X_1, \quad 
\od{X_2}{t} = (\eps + 3 \eps^2) X_2 + (1+\eps-\eps^2-2\eps^3) J X_2.
\]
The eigenvalues of the $k$-th order model can be shown to approximate the 
slow eigenvalues $\alpha_{+,\pm}$ of the full system with an error of 
$\cO(\eps^{k+1})$.

The system is solved numerically with $\eps = 0.05$ up to a long time 
$T = 25 > \eps^{-1}$ using the scheme in \cite{Jin2022HighOrder}. 
The numerical results are shown in \Cref{fig:naive}.

\begin{figure}[htp]
    \centering
    \begin{subfigure}{.45\linewidth}
        \label{fig:naive.sub.1}
        \includegraphics[width=\linewidth]{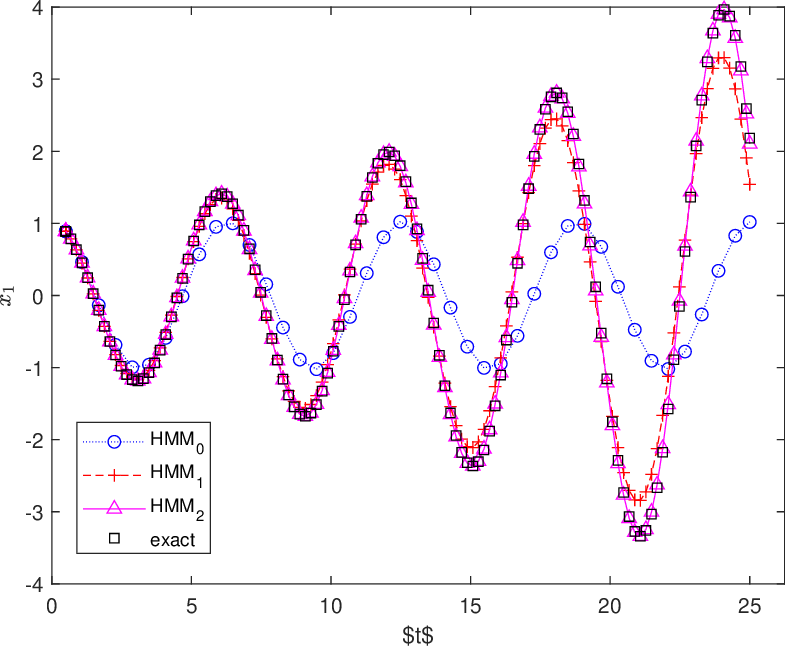}
        \caption{First element of slow variables.}
    \end{subfigure}
    \hfill
    \begin{subfigure}{.45\linewidth}
        \label{fig:naive.sub.2}
        \includegraphics[width=\linewidth]{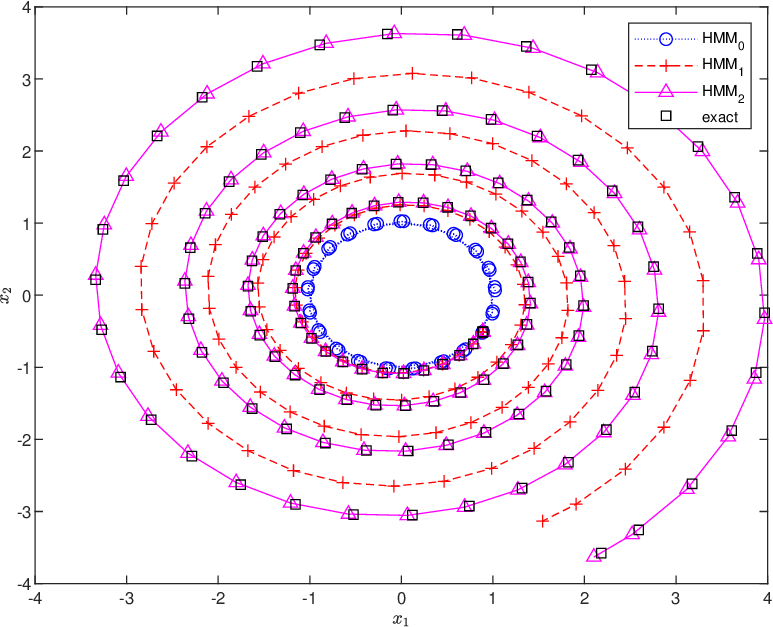}
        \caption{Trajectory of slow variables.}
    \end{subfigure}
    \caption{Numerical results for the linear model 
    \eqref{equ:model_failure} with $\eps = 0.05$ and $T = 25$.}
    \label{fig:naive}
\end{figure}

As predicted by \Cref{thm:longtime}, the error of the $k$-th order model 
accumulates linearly with a rate proportional to $\eps^{k+1}$. 
The trajectory of the classical HMM exhibits a significant phase lag and 
incorrect amplitude, drifting substantially from the true solution $x(t)$. 
This is a clear visual manifestation of the $\cO(1)$ accumulated error 
predicted for low-order methods over long times.
Even the first-order model shows visible deviation over time.
Only the second-order model remains close to the true solution for the 
entire simulation duration.
This provides an immediate validation of the theory, perfectly 
illustrating that higher-order models are essential for long-time accuracy.

\section{An efficient multigrid-in-time algorithm}
\label{sec:grid}

The preceding analysis establishes that high-order models are necessary 
for long-time accuracy. 
This section addresses the second major challenge: the redundant 
computational cost of direct high-order methods. 
An efficient multigrid-in-time (MGT) algorithm is introduced that achieves 
high-order accuracy at a cost comparable to low-order methods.

\subsection{Motivation and algorithm principle}

The numerical construction of the $k$-th order model involves a recursive 
procedure in \eqref{eq:iteration} and \eqref{eq:numder}, where computing 
$\Gamma_k$ requires information from $\Gamma_{k-1}$. 
This leads to a computational cost for the $k$-th order effective force 
$F_k(x)$ that scales as $\cO(2^k)$ relative to the classical HMM force 
$F_0(x)$. 
This exponential growth in complexity renders high-order models 
impractical for long-time simulations.

The key insight that enables the MGT approach is that the correction 
terms, $\delta_k(x) = F_k(x) - F_{k-1}(x)$, are not only small in 
magnitude, i.e., $\cO(\eps^k)$, but are also smooth functions of $x$.
As the slow variables $x(t)$ evolve smoothly in time, the correction term 
evaluated along the trajectory, $\delta_k(x(t))$, also varies smoothly 
with time.
The MGT framework systematically exploits this regularity to reduce 
redundant computation.
The core idea of the MGT algorithm is to use multiple time levels.
Instead of computing this expensive term at every fine time step, it is 
evaluated only on a coarse time grid and its effect is propagated to the 
fine grid via extrapolation.
This approach avoids redundant computations by reusing existing 
calculation results as much as possible. 

\subsection{The two-grid method}\label{sec:grid_two}
To illustrate the fundamental idea, we first describe a two-grid method.
The goal is to solve the expensive high-order model
\begin{equation}\label{eq:twogrid-high}
  \od{X}{t} = F(X),
\end{equation}
by decomposing the vector field as $F(X) = \hat{F}(X) + \delta(X)$, 
where $\hat{F}(X)$ is a cheaper, lower-order approximation and 
$\delta(X)$ is the expensive correction.
The algorithm proceeds as follows:
\begin{enumerate}
  \item Define a fine time grid with step size $\Delta t$ and a coarse 
  grid with step size $\tau = P \Delta t$ for some integer $P > 1$.
  \item At coarse time steps $t_n = n \tau$, the expensive correction term
  $\delta(\hat{X}(t_n))$ is accurately evaluated and stored.
  \item For fine steps $t \in (t_n, t_{n+1})$, use an ODE solver with step 
  size $\Delta t$ to evolve the cheaper dynamics, corrected by an 
  extrapolated value of the expensive term:
  \begin{equation}\label{eq:twogrid}
    \od{\hat{X}}{t} = \hat{F}(\hat{X}) + 
    \cE[\delta(\hat{X}(\cdot)); \Delta](t), \quad t \in (t_n, t_{n+1}),
  \end{equation}
  where $\Delta = \{ t_{n-m}, \ldots, t_n \}$ is a stencil of recent 
  coarse grid points and $\cE$ is a polynomial extrapolation operator
  defined in \eqref{eq:extrap}.
\end{enumerate}

\begin{remark}\label{rmk:initial_extrap}
To ensure enough data points are available for stable extrapolation, the 
full high-order model \eqref{eq:twogrid-high} must be solved accurately 
for a short initial period $t \in [0, t_m]$. 
In other words, 
\begin{equation}\label{eq:twogrid-initial}
\od{\hat{X}}{t} = F(\hat{X}), \quad t \in [0, t_m].
\end{equation}
\end{remark}

\begin{remark}\label{rmk:differ_extrap}
In practice, the high-order vector field takes the form of 
$F_{k+1}(X) = f(X, \Gamma_{k+1}(X, \eps))$.
In addition to the approach of direct correction of the vector field 
$F_k(X)$, one can also correct the $k$-th order invariant manifold 
$\Gamma_k(X, \eps)$.
In that case, the two-grid method evolves the system:
\[
\od{\hat{X}_k}{t} = f(\hat{X}_k, \Gamma_k(\hat{X}_k, \eps) + 
\cE[\Gamma_{k+1}(\hat{X}_k(\cdot), \eps) - 
\Gamma_k(\hat{X}_k(\cdot), \eps); \Delta](t)).
\]
The two methods have many similarities in terms of numerical analysis, 
complexity analysis, and algorithmic implementation. 
Therefore, we will focus on the method of directly correcting the vector 
field.
\end{remark}

\subsection{The full multigrid framework}
The two-grid concept generalizes naturally to a full multigrid hierarchy. 
To approximate the solution of an accurate $(k+L)$-th order model, 
$\od{X_{k+L}}{t} = F_{k+L}(X_{k+L})$, a hierarchy of $L+1$ nested time 
grids, $\cG_0 \supset \cG_1 \supset \cdots \supset \cG_L$, is introduced, 
with corresponding time steps 
$\Delta t = \tau_0 < \tau_1 < \cdots < \tau_L$, 
where $\tau_\ell$ is a multiple of $\tau_{\ell-1}$ for each 
$\ell = 1, 2, \ldots, L$.
Here $\cG_0$ is the grid where the macroscopic solver is used, and 
$\cG_1, \cG_2, \ldots, \cG_L$ are the correction layers.
The target vector field is decomposed via a telescoping sum:
\[
F_{k+L}(X) = F_k(X) + \sum_{\ell=1}^{L} 
(F_{k+\ell}(X) - F_{k+\ell-1}(X)) = F_k(X) + \sum_{\ell=1}^{L}
\delta_{k+\ell}(X).
\]
The MGT-HMM algorithm, denoted by $\HMM_k^L$, integrates the system on 
the finest grid $\cG_0$ with step $\Delta t$, using the base model $F_k$
and adding extrapolated corrections from all higher levels.
The correction $\delta_{k+\ell}$ is only computed at the nodes of the 
coarser grid $\cG_\ell$.
The resulting evolution equation is: 
\[
\od{\hat{X}}{t} = F_k(\hat{X}) + \sum_{\ell=1}^L 
\cE[\delta_{k+\ell}(\hat{X}(\cdot)); \Delta_\ell](t),
\]
where $\Delta_\ell \subset \cG_\ell$ is a stencil of $m_\ell+1$ recent 
data points on the grid $\cG_\ell$.

\subsection{Complexity analysis}\label{sec:grid_compl}
The computational cost of the MGT-HMM algorithm is dominated by the 
evaluations of the vector fields.
Let $\cC_j \approx \cC_0 2^j$ be the cost of evaluating $F_j(x)$.
The time grids are chosen to be geometric, $\tau_\ell = \tau_0 P^\ell$ for 
some integer $P > 1$.
The base model is $F_k$, and there are $L$ correction levels.
The total simulation time is $T$.
According to \eqref{eq:numder}, to evaluate $F_{k+\ell}(X)$, one needs to 
evaluate $F_{k+\ell-1}(X)$ first.
The most expensive vector field $F_{k+L}$, which is our aim to 
approximate, is evaluated only on the coarsest grid $\cG_L$.
For $\ell = 0, 1, \ldots, L -1$, the vector field $F_{k+\ell}$ needs to be 
evaluated on the grid $\cG_\ell \setminus \cG_{\ell+1}$.
The total computational cost is the sum of these contributions:
\[
\Cost_{\HMM_k^L} \approx |\cG_L| \cdot \cC_{k+L} + 
\sum_{\ell=0}^{L-1} |\cG_\ell \setminus \cG_{\ell+1}| \cdot \cC_{k+\ell}
\approx \frac{T}{\tau_L} \cC_{k+L} + \sum_{\ell=0}^{L-1} 
\left( \frac{T}{\tau_{\ell}} - \frac{T}{\tau_{\ell+1}} \right) 
\cC_{k+\ell}.
\]
Let $\Cost_{\HMM_k} \approx (T / \tau_0) \cC_k$. 
Normalizing this cost and substituting $\cC_{k+\ell} = \cC_k 2^\ell$ and 
$\tau_\ell = \tau_0 P^\ell$, one gets that 
\[
\Cost_{\HMM_k^L} \approx \Cost_{\HMM_k} \cdot 
\left( \left( \frac2P \right)^L + \sum_{\ell=0}^{L-1} 
\left( \frac2P \right)^\ell \cdot \left( 1 - \frac1P \right) \right).
\]
Evaluating the geometric series for two cases yields:
\[
\frac{\Cost_{\HMM_k^L} }{\Cost_{\HMM_k}} \approx
\begin{cases}
1+\frac{L}2, & P = 2, \\
\frac{P-1}{P-2} - \frac{1}{P-2} \left( \frac2P \right)^L, & P \ge 3.
\end{cases}
\]
This should be compared with the cost of a naive $(k+L)$-th order method,
which is $2^L$ times the cost of the base $k$-th order model:
\[
\Cost_{\HMM_{k+L}} \approx 2^L \cdot 
\Cost_{\HMM_k}.
\]
If the grid refinement factor $P > 2$ is chosen, the amplification of the 
computational cost from $\HMM_k$ to $\HMM_k^L$ is bounded by an $\cO(1)$ 
constant $\frac{P-1}{P-2}$.
Therefore, the total cost of $\HMM_k^L$ is of the same order as the cost 
of the base $k$-th order method $\HMM_k$.
This analysis demonstrates a dramatic efficiency gain over the naive 
approach.

\section{A priori error analysis of the MGT-HMM scheme}
\label{sec:analysis} 

This section provides a rigorous a priori error analysis for the proposed 
MGT-HMM algorithm, culminating in a practical guide for balancing the 
different error sources to achieve optimal efficiency.

\subsection{Properties of polynomial extrapolation operator}
\label{sec:analysis_extrap}
The core of the MGT-HMM algorithm is the use of polynomial extrapolation 
to approximate expensive high-order correction terms on a fine time grid 
using values computed only on a coarse grid. 
We begin by formalizing the properties of the extrapolation operator that 
are essential for the error analysis. 

Let $\Delta \subset \bR$ be a non-empty finite set of points, 
and let $h$ be a function defined on $\Delta$.
The nodal polynomial associated with $\Delta$ is defined as 
$\phi_\Delta(t) = \prod_{t' \in \Delta} (t - t')$.
The extrapolation operator of the function $h$ on $\Delta$ via a Lagrange 
polynomial is denoted by $\cE[h; \Delta]$, i.e.,
\begin{equation}\label{eq:extrap}
\cE[h; \Delta](t) := \sum_{t' \in \Delta} h(t') 
\frac{\phi_{\Delta \setminus \{ t' \}}(t)}
{\phi_{\Delta \setminus \{ t' \}}(t')}.
\end{equation}
The following lemma provides standard error bounds that are essential for 
the analysis.

\begin{lemma}\label{lemma:extrap}
Suppose that $\Delta \cup \{t\}$ is contained in a bounded closed interval 
$I \subset \bR$, and the function $h$ is defined on $I$. 
The following estimates hold:
\begin{enumerate}
\item If $h \in C^{m}(I)$, where $m = |\Delta|$ is the size of the set 
$\Delta$, then 
\[
|h(t) - \cE[h; \Delta](t)| \le \frac{1}{m!} |\phi_\Delta(t)| 
\max_{t' \in I} | h^{(m)}(t') |.
\]
\item The extrapolation operator $\cE$ is $L^\infty$-stable in the 
following sense: if $h \in L^\infty(I)$, then 
\[
\| \cE[h; \Delta] \|_{L^\infty(I)} \le \Lambda_I(\Delta) \cdot 
\|h\|_{L^\infty(I)},
\]
where $\Lambda_I(\Delta) = \sup_{t \in I} \sum_{t' \in \Delta} 
\left| \frac{\phi_{\Delta \setminus \{ t' \}}(t)}
{\phi_{\Delta \setminus \{ t' \}}(t')} \right|$ is the Lebesgue constant
for the stencil $\Delta$ on the interval $I$.
\item If $h$ is Lipschitz continuous on $I$, then 
\[
|h(t) - \cE[h; \Delta](t)| \le \| h \|_{Lip} \cdot 
|\phi_\Delta(t)| \cdot \sum_{t' \in \Delta} 
\frac{1}{|\phi_{\Delta \setminus \{ t' \}}(t')|}.
\]
\end{enumerate}
\end{lemma}

\begin{proof}
The first and second statements are standard results from elementary 
mathematical analysis.
For the third statement, we use the fact that the extrapolation operator 
is exact for constant functions, i.e., $\cE[1; \Delta] \equiv 1$. 
This implies that 
\begin{multline*}
|h(t) - \cE[h; \Delta](t)| = 
\left| \sum_{t' \in \Delta} (h(t) - h(t') ) 
\frac{\phi_{\Delta \setminus \{ t' \}}(t)}
{\phi_{\Delta \setminus \{ t' \}}(t')} \right| \\ \le 
\| h \|_{Lip} \sum_{t' \in \Delta} 
\left| \frac{\phi_{\Delta \setminus \{t'\}}(t) \cdot (t - t')}
{\phi_{\Delta \setminus \{t'\}}(t')} \right| \le \| h \|_{Lip} \cdot 
|\phi_\Delta(t)| \cdot \sum_{t' \in \Delta} 
\frac{1}{|\phi_{\Delta \setminus \{ t' \}}(t')|},
\end{multline*}
which is the desired result.
\end{proof}

\begin{remark}
In practical implementations, the extrapolation stencil typically consists 
of $m+1$ equispaced points.
Let $\Delta = \{t_0, t_1, \ldots, t_m\}$, where 
$t_k = t_0 + k \tau$ for some step size $\tau > 0$.
The size of $\Delta$ is $|\Delta| = m+1$.
For $t \in [t_m, t_{m+1})$, one has that 
\[
|\phi_\Delta(t)| = \prod_{k=0}^m |t - t_k| \le \prod_{k=0}^m 
|t_{m+1} - t_k| = (m+1)! \cdot \tau^{m+1},
\]
which yields that 
\begin{equation}\label{eq:equi_smooth}
|h(t) - \cE[h; \Delta](t)| \le \tau^{m+1} \max_{t' \in I} 
| h^{(m+1)}(t') |.
\end{equation}
In addition, it is known that the Lebesgue constant is scaling invariant, 
and it grows exponentially with respect to $m$ for equispaced points.
With Lipschitz continuity, one can obtain that 
\[
\frac{1}{|\phi_{\Delta \setminus \{ t' \}}(t')|} = 
\frac{1}{\tau^m} \sum_{k = 0}^m 
\prod_{j = 0 \atop j \ne k}^m \frac{1}{|j - k|} = \frac{1}{\tau^m} 
\sum_{k=0}^m \frac{1}{k! (m - k)!} = 
\frac{2^m}{m! \cdot \tau^m}.
\]
At this time, there is a better estimate:
\begin{equation}\label{eq:equi_Lip}
|h(t) - \cE[h; \Delta](t)| \le 2^m (m+1) \tau \cdot \|h\|_{Lip}.
\end{equation}
Here the exponential growth of the factor $\tilde{\Lambda}_m = 2^m (m+1)$ 
is similar to the behavior of the Lebesgue constant $\Lambda_m$, which 
motivates the use of low-to-moderate extrapolation orders in practice.
Compared to the $L^\infty$ case, the estimate in the Lipschitz case has an 
additional $\tau$ given by the scaling argument, which will bring us some 
benefits in the subsequent error estimate.
\end{remark}

\subsection{The two-grid method}
To build intuition, the error for the two-grid method is first analyzed. 
The error between the exact high-order solution $X(t)$ and the two-grid 
approximation $\hat{X}(t)$ is bounded in the following theorem.

\begin{theorem}\label{thm:twogrid}
Let $X(t)$ be the solution to \eqref{eq:twogrid-high} and $\hat{X}(t)$ be 
the solution to \eqref{eq:twogrid} with initial treatment in 
\eqref{eq:twogrid-initial}. 
Assume that the vector field $F$ is $C^{m+1}$ and satisfies that 
$\sinn{F(x) - F(\tx), x - \tx} \le \lambda |x - \tx|^2$. 
Let $T$ be the maximal stability time of \eqref{eq:twogrid-high} 
determined by $\lambda$.
Let $\delta = F - \hat{F}$ and assume that 
$\| \delta \|_{C^{m+1}} \le \eta$.
Assume that $\eta$ and $\tau$ are sufficiently small.
Then there exists a constant $C > 0$ independent of $\eta, \tau, m, t$ 
such that for $t \in [t_m, \min \{ T, \eta^{-1} \}]$, 
the error is bounded by 
\[
|X(t) - \hat{X}(t)| \le C (\eta \tau^{m+1} + \tilde{\Lambda}_m 
\eta^2 \tau) \cdot \zeta(t - t_m),
\]
where $\zeta(t) = \int_0^t e^{\lambda s} \,\rd s$ and 
$\tilde{\Lambda}_m = 2^m (m+1)$.
\end{theorem}

\begin{proof}
For fixed $n \ge m$ and $t \in [t_n, t_{n+1})$, the approximation error 
can be estimated by differentiating $|X - \hat{X}|^2 / 2$ as follows: 
\begin{multline*}
\frac12 \od{}{t} |{X - \hat{X}}|^2 = 
\sinn{F(X) - \hat{F}(\hat{X}) - 
\cE[F(\hat{X}(\cdot)) - \hat{F}(\hat{X}(\cdot)); \Delta], X - \hat{X}} \\ 
= \sinn{F(X) - F(\hat{X}), X - \hat{X}} + \sinn{\delta(\hat{X}) - 
\cE[\delta(\hat{X}(\cdot)); \Delta], X - \hat{X}} \\ 
\le \lambda \|{X - \hat{X}}\|^2 + \sinn{\delta(X) - 
\cE[\delta({X}(\cdot)); \Delta], X - \hat{X}} \\ 
+ \sinn{\delta(\hat{X}) - \delta(X) - 
\cE[\delta(\hat{X}(\cdot)) - \delta(X(\cdot)); \Delta], X - \hat{X}}.
\end{multline*}
The $C^{m+1}$ regularity of the vector field $F$ implies that the solution 
$X(t)$ has bounded derivatives up to order $m+2$, according to 
\cite[Lemma C.1 (2)]{Jin2022HighOrder}, yielding that 
$\| \delta \circ X \|_{C^{m+1}} \lesssim \eta$.
Applying \Cref{lemma:extrap}, one can bound the second term by 
\[
| \delta(X) - \cE[\delta({X}(\cdot)); \Delta] | \lesssim \eta \tau^{m+1}.
\]
In addition, one notes that 
\begin{gather*}
|\delta(\hat{X}(t)) - \delta(X(t))| + 
|\nabla \delta(\hat{X}(t)) - \nabla \delta(X(t))| 
\lesssim \eta |X(t) - \hat{X}(t)|, \\
\left|\nabla \delta(\hat{X}) \od{\hat{X}}{t} - 
\nabla \delta(X) \od{X}{t} \right| \le 
|\nabla \delta(\hat{X})| \cdot \left| \od{\hat{X}}{t} - \od{X}{t} \right| 
+ |\nabla \delta(\hat{X}) - \nabla \delta(X)| \cdot \left| \od{X}{t} 
\right|,
\end{gather*}
and 
\begin{multline*}
\left| \od{\hat{X}}{t} - \od{X}{t} \right| = 
|F(X) - \hat{F}(\hat{X}) - \cE[\delta(\hat{X}(\cdot)); \Delta]| \\ 
\le |F(X) - F(\hat{X})| + |F(\hat{X}) - \hat{F}(\hat{X})| 
+ |\cE[\delta(\hat{X}(\cdot)); \Delta]| 
\lesssim \eta + |X(t) - \hat{X}(t)|.
\end{multline*}
Therefore, 
\[
\|\delta(\hat{X}) - \delta(X)\|_{Lip} \lesssim \eta 
\max_{t' \in [t_{n-m}, t]} |X(t') - \hat{X}(t')| + \eta^2.
\]
The last term is then bounded using the stability property of the 
extrapolation operator, yielding a bound as follows:
\[
|\delta(\hat{X}) - \delta(X) - 
\cE[\delta(\hat{X}(\cdot)) - \delta(X(\cdot)); \Delta]| 
\lesssim \tilde{\Lambda}_m \eta \tau \max_{t' \in [t_{n-m}, t]} 
|X(t') - \hat{X}(t')| + \tilde{\Lambda}_m \eta^2 \tau.
\]
Combining these estimates, one can obtain the following inequality:
\begin{multline*}
\frac12 \od{}{t} |X - \hat{X}|^2 \\ \le \lambda |X - \hat{X}|^2 + 
C |X - \hat{X}| \cdot \left( \eta \tau^{m+1} + \tilde{\Lambda}_m \eta \tau 
\max_{t' \in [0, t]} |X(t') - \hat{X}(t')| + 
\tilde{\Lambda}_m \eta^2 \tau \right).
\end{multline*}
Therefore, by \Gronwall's inequality, one has 
\begin{multline*}
|X(t) - \hat{X}(t)| \le C \int_{t_m}^t e^{\lambda (t - s)} 
\left( \eta \tau^{m+1} + \tilde{\Lambda}_m \eta \tau \max_{t' \in [0, s]} 
|X(t') - \hat{X}(t')| + \tilde{\Lambda}_m \eta^2 \tau \right) \,\rd s \\
= C \tilde{\Lambda}_m \eta \tau \cdot \int_{t_m}^t e^{\lambda (t-s)} \cdot 
\max_{t' \in [0, s]} |X(t') - \hat{X}(t')| \,\rd s + 
C \left( \eta \tau^{m+1} + \tilde{\Lambda}_m \eta^2 \tau \right) \cdot 
\zeta(t - t_m).
\end{multline*}
for each $t \ge t_m$.
Let us define $\xi(t) = \max_{t' \in [0, t]} |X(t') - \hat{X}(t')|$, and 
assume that the maximum is attained at the point $\tilde{t}$.
Since $t \le T$ and $\zeta(t)$ is non-decreasing, one can obtain that 
\begin{multline*}
\xi(t) = |X(\tilde{t}) - \hat{X}(\tilde{t})| 
\\ \le C \tilde{\Lambda}_m \eta \tau \cdot \int_{t_m}^{\tilde{t}} 
e^{\lambda (\tilde{t}-s)} \cdot \xi(s) \,\rd s + 
C \left( \eta \tau^{m+1} + \tilde{\Lambda}_m \eta^2 \tau \right) \cdot 
\zeta(\tilde{t} - t_m) \\
\lesssim \tilde{\Lambda}_m \eta \tau \cdot \int_{t_m}^{t} \xi(s) \,\rd s + 
\left( \eta \tau^{m+1} + \tilde{\Lambda}_m \eta^2 \tau \right) \cdot 
\zeta(t - t_m) 
\end{multline*}
A direct application of \Gronwall's inequality implies that 
\[
\xi(t) \le C \left( \eta \tau^{m+1} + \tilde{\Lambda}_m 
\eta^2 \tau \right) \cdot \zeta(t - t_m) \cdot \exp 
\left( C \tilde{\Lambda}_m \eta \tau (t - t_m) \right).
\]
Note that $t \lesssim \eta^{-1}$, which ensures that 
$\exp \left( C \tilde{\Lambda}_m \eta \tau (t - t_m) \right) \lesssim 1$.
This completes the proof.
\end{proof}

\begin{remark}
The error bound in \Cref{thm:twogrid} provides significant insight into 
the behavior of the two-grid method.
It can be dissected into several components:
\begin{enumerate}
\item The $\eta \tau^{m+1}$ term represents the intrinsic error of the 
polynomial approximation, arising from truncating the Taylor series of the 
correction term along the true trajectory.
\item The $\tilde{\Lambda}_m \eta^2 \tau$ term arises from the stability 
properties of the extrapolation operator applied to the difference between 
the correction evaluated on the true and approximate trajectories.
In practice, the discussions in \Cref{sec:analysis_balance} suggest 
$m \le q - 1$, where $q$ is the order of the macroscopic solver, 
implying that $\tilde{\Lambda}_m$ should not be too large.
In addition, when the base model is $\HMM_k$, one has that 
$\eta = \cO(\eps^{k+1})$, and $\tau^{m+1} \lesssim \eps$.
At this time, $\eta \tau^{m+1}$ should be the leading term of the error.
\item The error accumulation factor $\zeta(t - t_m)$ also appears in the 
long-time modeling error estimate in \Cref{thm:longtime_lambda}, 
reflecting how errors are propagated by the underlying dynamics.
Its behavior for different $\lambda$ can be found in 
\Cref{thm:longtime_lambda}.
\end{enumerate}
\end{remark}

\subsection{Error analysis of the full MGT-HMM scheme}

We now present the main result on the error estimate for the full 
$\HMM_k^L$ scheme, which approximates the solution $X_{k+L}$ of the 
$(k+L)$-th order correction model.

\begin{theorem}\label{thm:mgt}
Let $X_{k+L}(t)$ be the exact solution of the $(k+L)$-th order model, and 
let $\hat{X}(t)$ be the solution of the $\HMM_k^L$ scheme with 
extrapolation orders $m_\ell$ on grids $\cG_\ell$ with steps $\tau_\ell$ 
for $\ell = 1, 2, \ldots, L$.
Assume that the regularity index $K$ satisfies $K \ge k + \ell + m_\ell$ 
for each $\ell = 1, 2, \ldots, L$.
Let $T_\eps$ be the maximal stability time of $\HMM_{k+L}$ model 
determined by $\lambda_{k+L}(\eps)$.
Then there exists a constant $C > 0$ independent of $\eps, \tau_\ell, t$
such that for each $t \le \min \{ T_\eps, \eps^{-k-1} \}$, the 
error is bounded by 
\[
|X_{k+L}(t) - \hat{X}(t)| \le C
\sum_{\ell = 1}^L (\eps^{k+\ell} \tau_\ell^{m_\ell+1} + 
\eps^{2k+2\ell} \tau_\ell) \cdot \zeta_{k+L}(\eps, t).
\]
\end{theorem}

\begin{proof}
After the initial treatment, the approximation error can be estimated 
by differentiating $|X_{k+L} - \hat{X}|^2 / 2$ as follows:
\begin{multline*}
\frac12 \od{}{t} |X_{k+L} - \hat{X}|^2 = 
\ip{F_{k+L}(X_{k+L}) - F_k(\hat{X}) - 
\sum_{\ell=1}^L \cE[\delta_{k+\ell}(\hat{X}(\cdot)); \Delta_\ell]}
{X_{k+L} - \hat{X}} \\
\le \lambda_{k+L}(\eps) \cdot |X_{k+L} - \hat{X}|^2 + 
\sum_{\ell=1}^L |{\delta_{k+\ell}(\hat{X}) - 
\cE[\delta_{k+\ell}(\hat{X}(\cdot)); \Delta_\ell]}| \cdot 
|{X_{k+L} - \hat{X}}|.
\end{multline*}
Each term $|{\delta_{k+\ell}(\hat{X}) - 
\cE[\delta_{k+\ell}(\hat{X}(\cdot)); \Delta_\ell]}|$ is bounded as in the 
proof of \Cref{thm:twogrid}.
Note that 
\begin{multline*}
\left| \delta_{k+\ell}(\hat{X}) - 
\cE[\delta_{k+\ell}(\hat{X}(\cdot)); \Delta_\ell] \right|
\le \left| \delta_{k+\ell}(X_{k+L}) - 
\cE[\delta_{k+\ell}(X_{k+L}(\cdot)); \Delta_\ell] \right| \\ + 
\left| \delta_{k+\ell}(X_{k+L}) - 
\delta_{k+\ell}(\hat{X}) - 
\cE[\delta_{k+\ell}(X_{k+L}(\cdot)) - \delta_{k+\ell}(\hat{X}(\cdot)); 
\Delta_\ell] \right|.
\end{multline*}
The first term can be bounded by the following estimate:
\[
\left| \delta_{k+\ell}(X_{k+L}) - 
\cE[\delta_{k+\ell}(X_{k+L}(\cdot)); \Delta_\ell] \right| \lesssim
\eps^{k+\ell} \tau_\ell^{m_\ell + 1},
\]
and the second term can be bounded by 
\begin{multline*}
\left| \delta_{k+\ell}(X_{k+L}) - 
\delta_{k+\ell}(\hat{X}) - 
\cE[\delta_{k+\ell}(X_{k+L}(\cdot)) - \delta_{k+\ell}(\hat{X}(\cdot)); 
\Delta_\ell] \right| \\ \lesssim
\eps^{k+\ell} \tau_\ell \cdot \xi(t) + 
\eps^{2k+2\ell} \tau_\ell,
\end{multline*}
where $\xi(t) = \max_{t' \in [0,t]} |X_{k+L}(t') - \hat{X}(t')|$.
Summing these contributions and applying \Gronwall's inequality, one can 
obtain similarly to the proof of \Cref{thm:twogrid} that 
\[
\xi(t) \lesssim \sum_{\ell = 1}^L \left( 
\eps^{k+\ell} \tau_\ell^{m_\ell+1} + \eps^{2k+2\ell} \tau_\ell \right)
\cdot \zeta_{k+L}(\eps, t) \cdot 
\exp \left( C \sum_{\ell=1}^L \eps^{k+\ell}
\tau_\ell t \right).
\]
Applying the condition that $t \le \eps^{-k-1}$ yields the desired result.
\end{proof}

\subsection{A practical guide to parameter selection}
\label{sec:analysis_balance}

The complexity and error analyses provide a practical guide for optimizing 
the MGT-HMM algorithm.
Let us use a uniform extrapolation order, $m_\ell = m$, and a geometric 
grid hierarchy, $\tau_\ell = P^\ell \tau_0$,
for each $\ell = 1, 2, \ldots, L$.
To achieve a target accuracy of $\cO(\eps^{k+L+1})$ of $\HMM_k^L$, three 
primary sources of error must be balanced.

\begin{enumerate}
    \item Modeling error: The underlying $(k+L)$-th order model has 
    modeling error of order $\cO(\eps^{k+L+1})$ as established in 
    \Cref{thm:longtime_lambda}.
    This sets the target accuracy.
    \item Discretization error: The $q$-th order macroscopic ODE solver 
    introduces a discretization error of $\cO(\tau_0^q)$.
    To match the target accuracy, we require that 
    $\tau_0^q \sim \eps^{k+L+1}$.
    \item Extrapolation error: For each correction level 
    $\ell = 1, 2, \ldots, L$, the extrapolation error from \Cref{thm:mgt}
    should also be consistent with the target accuracy:
    \[
    \eps^{k+\ell} \tau_\ell^{m_\ell+1} + \eps^{2k+2\ell} \tau_\ell 
    \lesssim \eps^{k+L+1}.
    \]
    Balancing the first error term, we set $P \sim \eps^{- \frac{1}{m+1}}$ 
    and $\tau_0 \lesssim \eps^{\frac{L+1}{m+1}}$.
    Substituting these into the discretization error condition gives a 
    consistency requirement on the orders: 
    $m + 1 \sim \frac{q (L+1)}{k+L+1}$, 
    which for reasonable choices implies $m \le q - 1$.
    Balancing the second term yields a requirement that 
    $k \ge \frac{m}{m+1} L - 1$.
\end{enumerate}

In summary, to achieve $\cO(\eps^{k+L+1})$ modeling error with optimal 
computational cost, one needs:
\[
m_\ell = m, \quad \tau_\ell \lesssim \eps^{\frac{L+1-\ell}{m+1}}, \quad 
m + 1 \sim \frac{q (L+1)}{k+L+1}, \quad k \ge \frac{m}{m+1} L - 1.
\]

Now we compare the optimal computational cost between $\HMM_k^L$ and 
naive $\HMM_{k+L}$ to achieve the same modeling error $\cO(\eps^{k+L+1})$.
For the naive $\HMM_{k+L}$, balancing the modeling error and 
discretization error requires $\tau_0 \sim \eps^{\frac{k+L+1}{q}}$. 
The computational cost is 
$\Cost_{\HMM_{k+L}} \approx \frac{T \cC_{k+L}}{\tau_0}$.
For the MGT scheme $\HMM_k^L$ with $P \sim \eps^{-1/(m+1)}$, the 
computational cost is given by the analysis in \Cref{sec:grid_compl}:
\[
\Cost_{\HMM_k^L} \approx \frac{T \cC_k}{\tau_0} \cdot 
\left( \frac{P-1}{P-2} - \frac{1}{P-2} \left(\frac{2}{P}\right)^L \right).
\]
The ratio of costs is therefore
\begin{multline*}
\frac{\Cost_{\HMM_k^L}}{\Cost_{\HMM_{k+L}}} \approx 
\frac{P-1}{P-2} \cdot 2^{-L} - \frac{1}{P-2} \cdot P^{-L} \sim 
\frac{1 - \eps^{\frac{1}{m+1}}}{1 - 2 \eps^{\frac{1}{m+1}}} \cdot 2^{-L} - 
\frac{\eps^{\frac{L+1}{m+1}}}{1 - 2 \eps^{\frac{1}{m+1}}}  \\
\lesssim \left( 1+2\eps^{\frac{1}{m+1}} \right) \cdot 2^{-L} \approx 
\left( 1+2\eps^{\frac{1}{m+1}} \right) \cdot
\frac{\Cost_{\HMM_k}}{\Cost_{\HMM_{k+L}}},
\end{multline*}
as long as $\eps$ is sufficiently small. 
This can be rewritten as
\[
\Cost_{\HMM_k^L} \approx (1 + 2 \eps^{\frac{1}{m+1}}) \cdot \Cost_{\HMM_k}.
\]
This analysis demonstrates that the MGT-HMM scheme can achieve the 
accuracy of a $(k+L)$-th order model at a computational cost that is 
approximately $2^L$ times smaller than the naive approach, with only a 
marginal overhead of order $\cO(\eps^{\frac{1}{m+1}})$ compared to the 
base $k$-th order method.
This effectively decouples the pursuit of high accuracy from prohibitive 
computational expense.

\section{Numerical Experiments}
\label{sec:num}

This section presents a series of numerical experiments to provide 
definitive validation of the theoretical claims of accuracy and efficiency 
for the proposed MGT-HMM algorithm.
The setup follows that described in \cite{Jin2022HighOrder}: a standard 
4th-order Runge-Kutta scheme ($q = 4$) is used for the macroscopic solver 
with step size $\Delta t$ and for the coupled solver with step size 
$\Delta t_c$ in the initial layer.
Accuracy is quantified through the $\ell^2$-norm error of the slow variables 
at the final time $T$.

\subsection{Efficiency and accuracy}
The first experiment provides a direct demonstration of the MGT-HMM's 
efficiency gain over naive high-order methods without sacrificing 
accuracy. 
Consider the cubic Chua model \cite{Chua1986The} with $\eps = 0.02$: 
\[
\left\{
\begin{aligned} 
& \od{x_1}{t} = -x_2, \\
& \od{x_2}{t} = x_1 + a x_2 - by, \\
& \od{y}{t} = \frac{1}{\eps}(x_2 - c_3 y^3 - c_2 y^2 - c_1 y), 
\end{aligned}
\right.
\]
with $x_0 = (0.2, 0)$, $y_0 = 0$, $a = 0.1$, $b = 0.7$, 
$c_1 = 11$, $c_2 = \tfrac{41}{2}$, $c_3 = \tfrac{44}{3}$.
The MGT parameters are chosen as $P = 5$ and $m = 3$.
Other parameters are $\Delta t = 0.02$, $\Delta t_c = 10^{-4}$, 
$\eta = 10^{-4}$.
The micro-solver is a forward Euler ODE solver with time step size 
$\delta t = 0.1 \eps$ over $10$ steps in total.
The numerical results, presented in \Cref{fig:cubic}, compare the 
computational time and error at different system times $t$, 
for the naive $\HMM_k$ with $k = 0, 1, 2$ against the efficient 
two-grid $\HMM_k^L$ with $L = 1$ and $k = 0, 1, 2$.
As shown in \Cref{fig:cubic.sub.1}, the computational costs are clearly 
stratified.
The cost of $\HMM_k^1$ is only slightly higher than that of $\HMM_k$.
In contrast, the cost of $\HMM_{k+1}$ is approximately double the cost of 
$\HMM_k$.
\Cref{fig:cubic.sub.2} shows the corresponding errors.
The error curves for $\HMM_0^1$ and $\HMM_1$ are nearly identical, 
as are the curves for $\HMM_1^1$ and $\HMM_2$.
The combination of these two observations provides definitive proof of the 
method's efficiency and accuracy.
The $\HMM_0^1$ scheme achieves the accuracy of the first-order $\HMM_1$ 
scheme at a cost comparable to the zeroth-order $\HMM_0$ scheme, 
representing a speedup of nearly a factor of two.
Similarly, $\HMM_1^1$ achieves the accuracy of $\HMM_2$ at the cost of 
$\HMM_1$.
This confirms that the MGT framework successfully delivers high-order 
accuracy for the cost of a low-order method.

\begin{figure}[htp]
    \centering
    \begin{subfigure}[b]{0.49\textwidth}
        \centering
        \includegraphics[width=\linewidth]{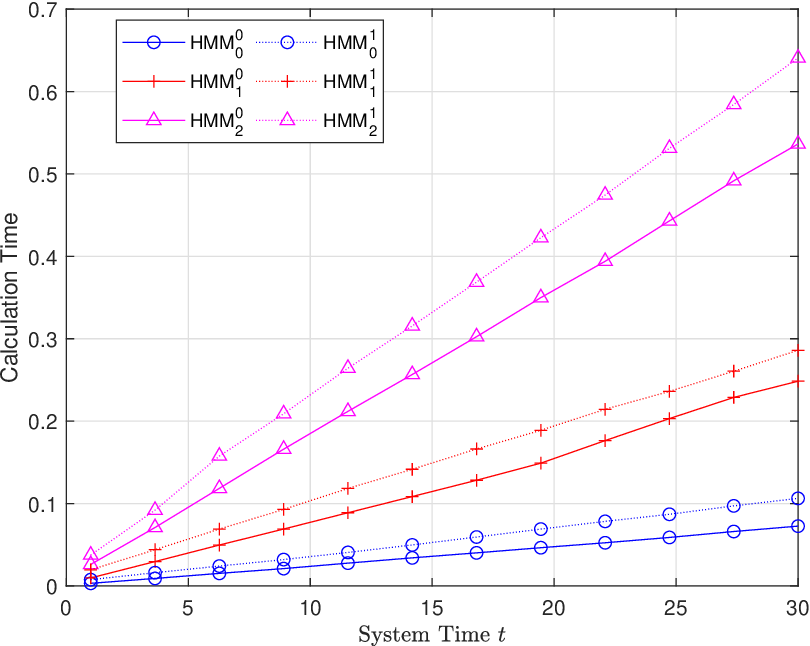}
        \caption{Calculation time versus system time.}
        \label{fig:cubic.sub.1}
    \end{subfigure}
    \hfill
    \begin{subfigure}[b]{0.49\textwidth}
        \centering
        \includegraphics[width=\linewidth]{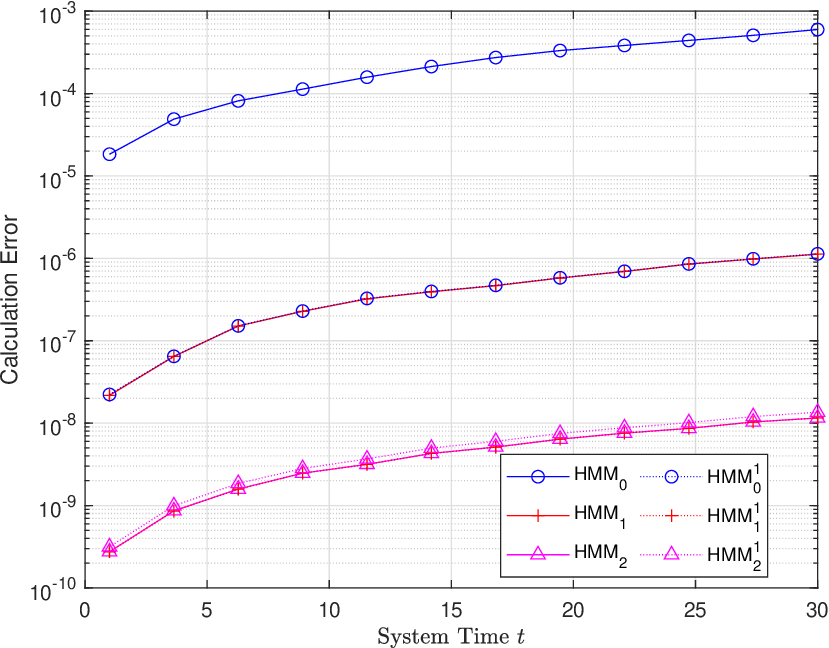}
        \caption{Calculation error versus system time.}
        \label{fig:cubic.sub.2}
    \end{subfigure}
    \caption{Comparison of computational cost and accuracy for the cubic 
    Chua model.}
    \label{fig:cubic}
\end{figure}

\subsection{Long-time modeling error}
This set of experiments is designed to validate the theoretical error 
scaling predictions from \Cref{sec:model} and \Cref{sec:analysis}.
To isolate the modeling error, a small time step size is used.

First, consider a simple linear model: 
\[
\left\{
\begin{aligned} 
& \od{x}{t} = - J y, \\
& \od{y}{t} = \frac{1}{\eps}(x - y), 
\end{aligned}
\right.
\]
where $x,y \in \bR^2$, $J = \left( \begin{smallmatrix}
0 & 1 \\ -1 & 0 \end{smallmatrix} \right)$, $x_0 = (1, 0)$, 
$y_0 = (0.9, 0.1)$.
We solve the system until $T_\eps = \eps^{-1}$.
The MGT parameters are $P = 4$ and $m = 3$. 
Other parameters are $\Delta t = 10^{-3}$, $\Delta t_c = 10^{-4}$, 
$\eta = 5 \times 10^{-6}$.
The micro-solver is a forward Euler ODE solver with time step size 
$\delta t = \eps$ over $10$ steps in total.
\Cref{fig:linear} shows the numerical error at $t = T_\eps$ for the 
two-grid $\HMM_k^1$ scheme with $k = 0, 1, 2$.
The target model is $\HMM_{k+1}$, whose modeling error is 
$\cO(\eps^{k+2} t)$.
At $t = T_\eps = \eps^{-1}$, the accumulated error is theoretically 
$\cO(\eps^{k+1})$.
The numerical results in \Cref{fig:linear}, showing the error versus 
$\eps$ on a log-log scale, perfectly match this theoretical expectation, 
with observed slopes corresponding to orders $\cO(\eps^1)$, $\cO(\eps^2)$, 
and $\cO(\eps^3)$ for $k = 0, 1, 2$, respectively.

Second, consider the Lorenz-96 model \cite{Lorenz96}: 
\[
\left\{
\begin{aligned} 
& \od{x_k}{t} = -x_{k-1} (x_{k-2} - x_{k+1}) - x_k + a - h b^{-1}
\sum_{j = 1}^{J} y_{j,k}, \\
& \od{y_{j,k}}{t} = \frac{1}{\eps}\left( - b y_{j+1,k} 
(y_{j+2,k} - y_{j-1,k}) - y_{j,k} + h b^{-1} x_k \right), 
\end{aligned}
\right.
\]
where $j = 1, 2, \ldots, J$, $k = 1, 2, \ldots, K$.
The subscripts $j$ and $k$ are periodic, modulo $J$ and $K$, respectively.
We take $J = 10$, $K=36$, $a = 1$, $b = 10$, and $h = K^{-1}$.
The initial conditions are set to $x_k(0) = z_k (z_k-1) (z_k+1)$ where 
$z_k = \tfrac{1+2k}{K} - 1$ for $k = 1,2, \ldots, K$, and $y_{j,k}(0) = 0$ 
for all $j,k$.
The system is solved until $T = 100$.
The MGT parameters are $P = 4$ and $m = 3$.
Other parameters are $\Delta t = 10^{-3}$, $\Delta t_c = 10^{-4}$, 
$\eta = 10^{-6}$.
The micro-solver is a forward Euler ODE solver with time step size 
$\delta t = 0.5 \eps$ over 20 steps in total.
\Cref{fig:lorenz} shows the numerical error at $T = 100$ for the naive 
$\HMM_0$ (denoted $L=0$) and the $\HMM_0^{L}$ schemes with $L = 1, 2$, all 
using $\HMM_0$ as the base model. 
The results confirm the theoretical modeling error of order 
$\cO(\eps^{L+1})$, with the observed slopes matching the expected orders.

\begin{figure}[t]
	\begin{minipage}[t]{.49\linewidth}
		\centering
    \includegraphics[width=\linewidth]{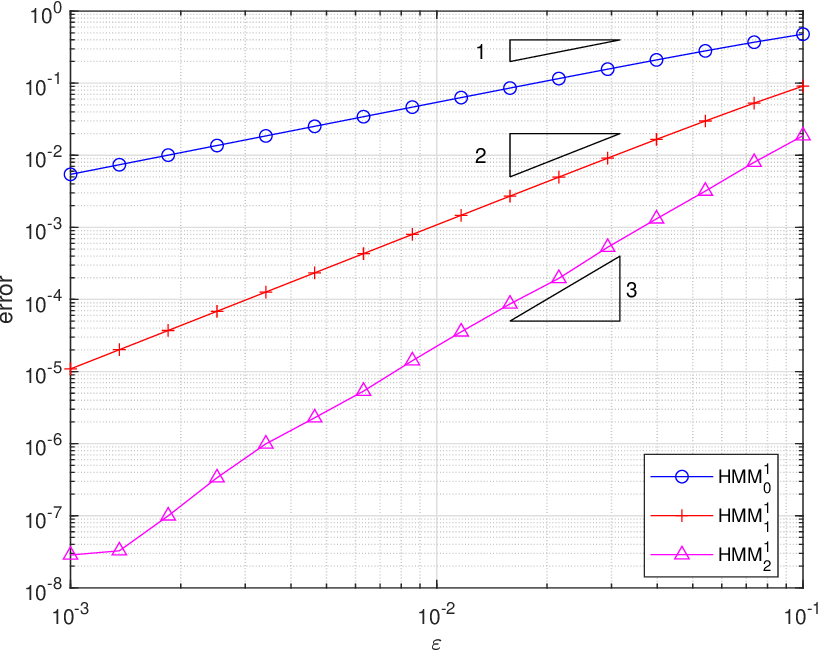}
    \caption{Numerical error at $T_\eps = \eps^{-1}$ versus $\eps$ for the 
    linear model using $\HMM_k^1$ schemes.}
    \label{fig:linear}
	\end{minipage}
  \hfill
	\begin{minipage}[t]{.49\linewidth}
		\centering
    \includegraphics[width=\linewidth]{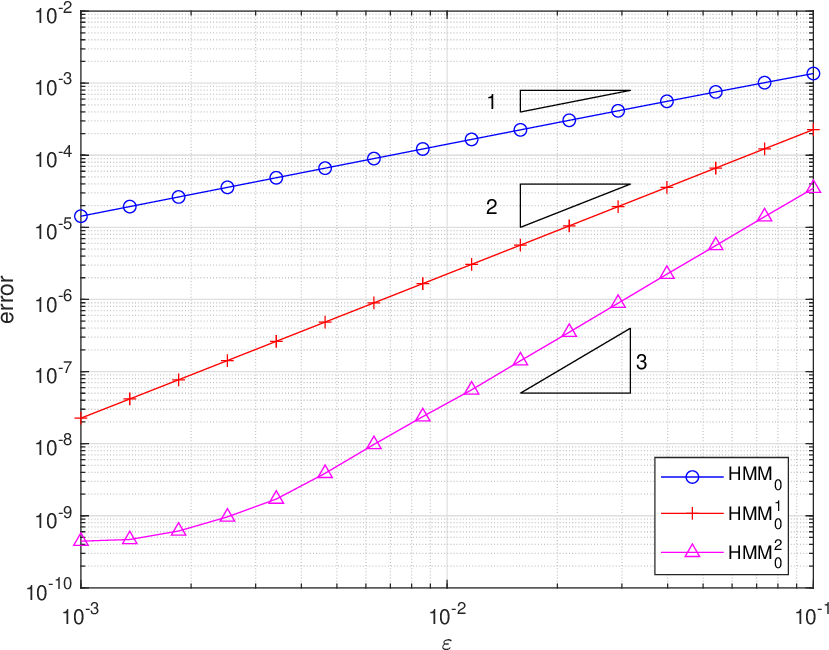}
    \caption{Numerical error at $T = 100$ versus $\eps$ for the 
    Lorenz-96 model using $\HMM_0^L$ schemes.}
    \label{fig:lorenz}
	\end{minipage}
\end{figure}

Next, consider the Robertson model of chemical reaction, a prominent 
example of a stiff ODE system \cite{robertson1966solution}:
\[
\left\{
\begin{aligned} 
& \od{x}{t} = -0.04 x + (1-x-\eps y)y, \\
& \od{y}{t} = \frac{1}{\eps}(0.04 x - (1-x-\eps y)y-0.3 y^2), 
\end{aligned}
\right.
\]
with $x_0 = 1$, $y_0 = 0$, $T = 500$.
The MGT parameters are $P = 4$ and $m = 3$. 
Other parameters are $\Delta t = 10^{-3}$, $\Delta t_c = 10^{-4}$, 
$\eta = 10^{-5}$.
The micro-solver is an exact solver for the quadratic equation in $y$.
Robertson model has a dissipation in $x$, leading to a modeling error of 
order $\cO(\eps^{k+L+1})$ according to \Cref{thm:longtime_lambda}.
\Cref{fig:robertson} shows the numerical error at $t = T$ for the 
two-grid $\HMM_k^1$ with $k = 0, 1$. 
The numerical errors for $\HMM_0^1$ and $\HMM_1^1$ perfectly match the 
theoretical modeling error scaling of $\cO(\eps^2)$ and $\cO(\eps^3)$, 
respectively.

Finally, consider the enzyme reaction equation 
\cite{MR635782,Heineken1967On}:
\[
\left\{
\begin{aligned} 
& \od{x}{t} = -x + (x+c)y, \\
& \od{y}{t} = \frac{1}{\eps}(x - (x+1) y), 
\end{aligned}
\right.
\]
with $x_0 = 1$, $y_0 = 0$, $c = 0.5$, $T = 20$.
The MGT parameters are $P = 4$ and $m = 3$. 
Other parameters are $\Delta t = 10^{-3}$, $\Delta t_c = 10^{-4}$, 
$\eta = 5 \times 10^{-6}$.
The micro-solver is an exact solver for the linear equation in $y$.
The enzyme reaction model has a strong dissipation in $x$.
\Cref{fig:enzyme} shows the numerical error at $t = T$ for the 
$\HMM_0^1$, $\HMM_1^1$, and $\HMM_0^2$ schemes.
This experiment validates the full error scaling $\cO(\eps^{k+L+1})$.
The error for $\HMM_0^1$ ($k = 0, L = 1$) is $\cO(\eps^2)$, while the error 
for both $\HMM_1^1$ ($k = 1, L = 1$) and $\HMM_0^2$ ($k = 0, L = 2$) is 
$\cO(\eps^3)$, which completely satisfies the theoretical estimate.
In addition, the results compare the two extrapolation schemes mentioned 
in \Cref{rmk:differ_extrap}, i.e., correcting the manifold $\Gamma$ or the 
vector field $F$.
The numerical results for both correction strategies are almost identical, 
as claimed in \Cref{rmk:differ_extrap}, demonstrating the robustness of 
the MGT concept.

\begin{figure}[t]
	\begin{minipage}[t]{.49\linewidth}
		\centering
		\includegraphics[width=\linewidth]{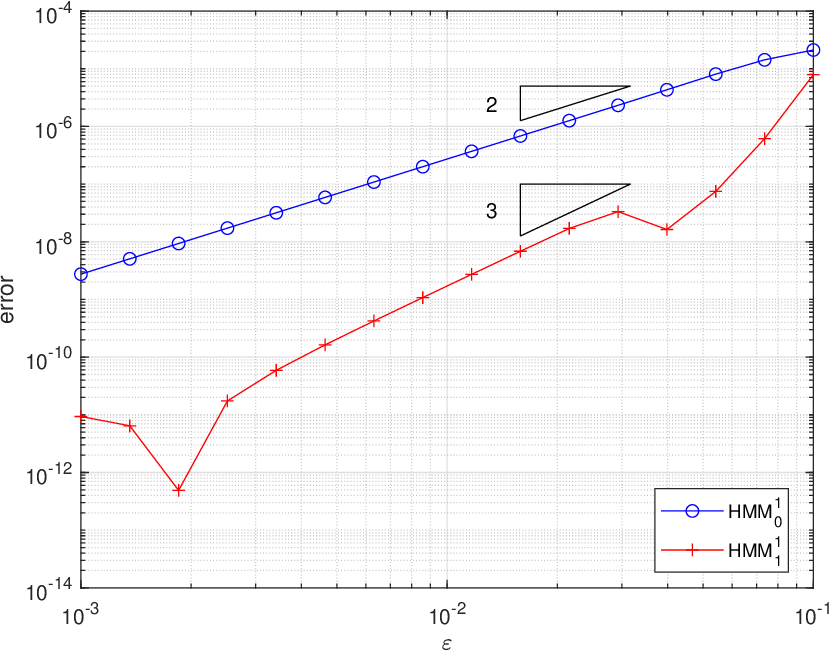}
    \caption{Numerical error at $T = 500$ versus $\eps$ for the Robertson model using $\HMM_k^1$ schemes.}
    \label{fig:robertson}
	\end{minipage}
  \hfill
	\begin{minipage}[t]{.49\linewidth}
		\centering
		\includegraphics[width=\linewidth]{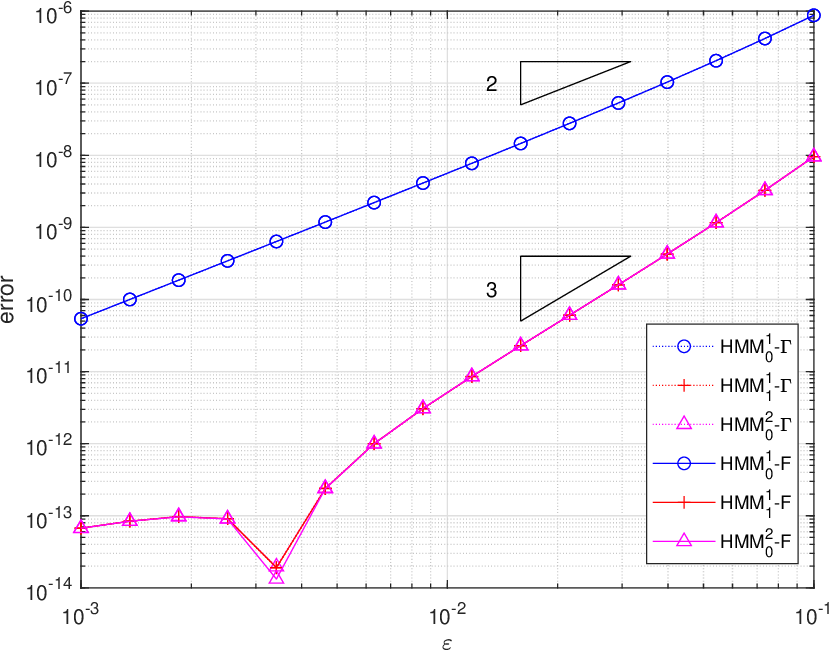}
    \caption{Numerical error at $T = 20$ versus $\eps$ for the enzyme 
    reaction model using various $\HMM_k^L$ schemes and correction 
    strategies.}
    \label{fig:enzyme}
	\end{minipage}
\end{figure}

\section{Conclusions}
\label{sec:concl}

In this work, we have addressed two fundamental obstacles to the efficient 
and accurate long-time simulation of multiscale ODEs: the lack of a 
rigorous long-time error theory for high-order methods and their redundant 
computational cost.

Our first major contribution was to close the analysis gap. 
We proved that for a large class of rapidly dissipative systems, the 
modeling error of the $k$-th order numerical homogenization method 
accumulates linearly in time, yielding a global error bound of 
$\cO(t \eps^{k+1})$. 
This result replaces the previously known exponential bound and provides a 
solid theoretical foundation for the use of high-order HMM in long-time 
simulations, confirming that higher-order models are indeed more accurate 
and reliable in the long-time regime.

Our second contribution was to bridge the efficiency gap. 
We introduced a novel multigrid-in-time HMM algorithm that systematically 
exploits the smoothness of high-order corrections to the effective 
dynamics. 
By computing these expensive corrections only on coarse time grids and 
extrapolating their influence to the fine grid, the MGT-HMM achieves the 
accuracy of a high-order model (e.g., order $k+L$) at a computational cost 
comparable to that of a low-order base model (order $k$). 
Our complexity analysis and numerical experiments confirm that this 
approach effectively decouples high accuracy from high computational cost,
reducing the expense by a factor of approximately $2^L$ compared to a 
naive implementation.

The combination of a rigorous long-time error theory and an efficient MGT 
algorithm makes high-order, long-time simulations of multiscale ODEs 
computationally feasible. 
This synergy, where the theory justifies the need for the algorithm and 
the algorithm makes the theoretically superior method practical, opens the 
door for more accurate and reliable modeling in the numerous scientific 
domains where such challenging multiscale problems arise.


\begin{thebibliography}{10}

\bibitem{Weinan2012The}
{\sc A.~Abdulle, W.~E, B.~Engquist, and E.~Vanden-Eijnden}, {\em The heterogeneous multiscale method}, Acta Numerica, 21 (2012), pp.~1--87.

\bibitem{abraham1997molecular}
{\sc F.~F. Abraham, D.~Brodbeck, W.~E. Rudge, and X.~Xu}, {\em A molecular dynamics investigation of rapid fracture mechanics}, Journal of the Mechanics and Physics of Solids, 45 (1997), pp.~1595--1619.

\bibitem{ascher1997implicit}
{\sc U.~M. Ascher, S.~J. Ruuth, and R.~J. Spiteri}, {\em Implicit-explicit {R}unge--{K}utta methods for time-dependent partial differential equations}, Applied Numerical Mathematics, 25 (1997), pp.~151--167.

\bibitem{Car1985Unified}
{\sc R.~Car and M.~Parrinello}, {\em Unified approach for molecular dynamics and density-functional theory}, Physical Review Letters, 55 (1985), p.~2471.

\bibitem{MR635782}
{\sc J.~Carr}, {\em Applications of centre manifold theory}, vol.~35 of Applied Mathematical Sciences, Springer-Verlag, New York-Berlin, 1981.

\bibitem{Chua1986The}
{\sc L.~Chua, M.~Komuro, and T.~Matsumoto}, {\em The double scroll family}, IEEE Transactions on Circuits and Systems, 33 (1986), pp.~1072--1118.

\bibitem{Weinan2003Analysis}
{\sc W.~E}, {\em Analysis of the heterogeneous multiscale method for ordinary differential equations}, Communications in Mathematical Sciences, 1 (2003), pp.~423--436.

\bibitem{Weinan2012Review}
{\sc W.~E}, {\em The heterogeneous multiscale method: A ten-year review}, in American Physical Society, 2012.

\bibitem{elliott2011novel}
{\sc J.~A. Elliott}, {\em Novel approaches to multiscale modelling in materials science}, International Materials Reviews, 56 (2011), pp.~207--225.

\bibitem{frank1997stability}
{\sc J.~Frank, W.~Hundsdorfer, and J.~G. Verwer}, {\em On the stability of implicit-explicit linear multistep methods}, Applied Numerical Mathematics, 25 (1997), pp.~193--205.

\bibitem{MR2759609}
{\sc M.~Haragus and G.~Iooss}, {\em Local bifurcations, center manifolds, and normal forms in infinite-dimensional dynamical systems}, Universitext, Springer-Verlag London, Ltd., London; EDP Sciences, Les Ulis, 2011, \url{https://doi.org/10.1007/978-0-85729-112-7}.

\bibitem{Heineken1967On}
{\sc F.~G. Heineken, H.~M. Tsuchiya, and R.~Aris}, {\em On the mathematical status of the pseudo-steady state hypothesis of biochemical kinetics}, Mathematical Biosciences, 1 (1967), pp.~95--113.

\bibitem{hochbruck2010exponential}
{\sc M.~Hochbruck and A.~Ostermann}, {\em Exponential integrators}, Acta Numerica, 19 (2010), pp.~209--286.

\bibitem{Jin2022HighOrder}
{\sc Z.~Jin and R.~Li}, {\em High-order numerical homogenization for dissipative ordinary differential equations}, Multiscale Model. Simul., 20 (2022), pp.~583--617, \url{https://doi.org/10.1137/21M1397003}.

\bibitem{MR2041455}
{\sc I.~G. Kevrekidis, C.~W. Gear, J.~M. Hyman, P.~G. Kevrekidis, O.~Runborg, and C.~Theodoropoulos}, {\em Equation-free, coarse-grained multiscale computation: enabling microscopic simulators to perform system-level analysis}, Commun. Math. Sci., 1 (2003), pp.~715--762, \url{https://doi.org/10.4310/cms.2003.v1.n4.a5}.

\bibitem{MR715971}
{\sc H.-W. Knobloch and B.~Aulbach}, {\em The role of center manifolds in ordinary differential equations}, in Equadiff 5 ({B}ratislava, 1981), vol.~47 of Teubner-Texte Math., Teubner, Leipzig, 1982, pp.~179--189.

\bibitem{Laskar1994Large}
{\sc J.~Laskar}, {\em Large-scale chaos in the solar system}, Astronomy and Astrophysics, 287 (1994), pp.~L9--L12.

\bibitem{Lorenz96}
{\sc E.~Lorenz}, {\em Predictability: a problem partly solved}, Seminar on Predictability, 1 (1995), pp.~1--18.

\bibitem{MR2922369}
{\sc A.~J. Majda}, {\em Challenges in climate science and contemporary applied mathematics}, Comm. Pure Appl. Math., 65 (2012), pp.~920--948, \url{https://doi.org/10.1002/cpa.21401}, \url{https://doi.org/10.1002/cpa.21401}.

\bibitem{Papanicolaou1976Some}
{\sc G.~C. Papanicolaou}, {\em Some probabilistic problems and methods in singular perturbations}, The Rocky Mountain Journal of Mathematics,  (1976), pp.~653--674.

\bibitem{pavliotis2008multiscale}
{\sc G.~Pavliotis and A.~Stuart}, {\em Multiscale methods: averaging and homogenization}, Springer Science \& Business Media, 2008.

\bibitem{robertson1966solution}
{\sc H.~Robertson}, {\em The solution of a set of reaction rate equations}, Numerical analysis: an introduction, 178182 (1966), p.~31.

\bibitem{Hairer1996Solving}
{\sc G.~Wanner and E.~Hairer}, {\em Solving ordinary differential equations {II}}, vol.~375, Springer Berlin Heidelberg, 1996.

\bibitem{Zhang1999A}
{\sc Y.~Zhang, T.-S. Lee, and W.~Yang}, {\em A pseudobond approach to combining quantum mechanical and molecular mechanical methods}, The Journal of Chemical Physics, 110 (1999), pp.~46--54.

\end{thebibliography}
\end{document}